\documentclass[12pt,a4paper]{amsart}
\usepackage{graphicx}
\usepackage{amssymb,amstext,amsmath,amsthm}
\usepackage[utf8]{inputenc}
\usepackage{color}
\usepackage[multiple]{footmisc}
\usepackage[normalem]{ulem}

\newenvironment{frcseries}{\fontfamily{frc}\selectfont}{}

\newcommand{\e}{e}
\newcommand{\eps}{\varepsilon}

\def\eps{\varepsilon}

\def\IN{\mathbb{N}}

\newtheorem{thm}{Theorem}
\newtheorem{corollary}{Corollary}
\newtheorem{lemma}{Lemma}
\newtheorem{pro}{Proposition}

\newtheorem*{algorithm}{Algorithm}

\theoremstyle{definition}

\theoremstyle{remark}
\newtheorem{remark}{Remark}
\newtheorem{xmpl}{Example}
\newtheorem{fact}{Fact}

\title[Singular value decomposition of integration operators]{On the singular value decomposition of n-fold integration operators}
\author{Ronny Ramlau}
\address{Institute for Industrial Mathematics, Johannes Kepler University Linz, Austria, and Johann Radon Institute for Computational and Applied Mathematics (RICAM),
  Austrian Academy of Sciences,
  Altenberger Stra\ss e 69, 4040 Linz, Austria}
\email{ronny.ramlau@jku.at}
\author{Christoph Koutschan}
\address{Johann Radon Institute for Computational and Applied Mathematics (RICAM),
  Austrian Academy of Sciences,
  Altenberger Stra\ss e 69, 4040 Linz, Austria}
\email{christoph.koutschan@ricam.oeaw.ac.at}
\author{Bernd Hofmann}
\address{Department of Mathematics, Chemnitz University of Technology,
  09107 Chemnitz,  Germany}
\email{bernd.hofmann@mathematik.tu-chemnitz.de}

\date{}

\begin{document}
\begin{abstract}
In theory and practice of inverse problems, linear operator equations $Tx=y$ with compact linear forward operators $T$ having a non-closed range $\mathcal{R}(T)$ and mapping between infinite dimensional Hilbert spaces plays some prominent role. As a consequence of the ill-posedness
of such problems, regularization approaches are required, and due to its unlimited qualification spectral cut-off is an appropriate method for the stable approximate solution of corresponding inverse problems.
For this method, however, the singular system  $\{\sigma_i(T),u_i(T),v_i(T)\}_{i=1}^\infty$ of the compact operator $T$ is needed, at least for $i=1,2,...,N$, up to some stopping index $N$. In this note we
consider $n$-fold integration operators $T=J^n\;(n=1,2,...)$ in $L^2([0,1])$ occurring in numerous applications, where the solution of the associated operator equation is characterized by the $n$-th generalized derivative $x=y^{(n)}$ of the Sobolev space function $y \in H^n([0,1])$. Almost all textbooks on linear inverse problems present the whole singular system $\{\sigma_i(J^1),u_i(J^1),v_i(J^1)\}_{i=1}^\infty$ in an explicit manner. However, they do not discuss the singular systems for $J^n,\;n \ge 2$. We will emphasize that this seems to be a consequence of the fact that for higher $n$ the eigenvalues $\sigma^2_i(J^n)$ of the associated ODE boundary
value problems obey transcendental equations, the complexity of which is growing with $n$. We present the transcendental equations for $n=2,3,...$ and discuss and illustrate the associated eigenfunctions and some of their properties.
\end{abstract}

\maketitle

\section{Introduction}
\label{se1}

For the stable approximate solution of the ill-posed linear operator equation
\begin{equation}\label{eq:opeq}
Tx=y, \qquad y \in \mathcal{R}(T),
\end{equation}
with a compact linear operator $T$ mapping between the infinite dimensional Hilbert spaces $X$ and $Y$ with norms $\|\cdot\|$ and inner products $\langle \cdot,\cdot\rangle$ and range $\mathcal{R}(T) \not= \overline{\mathcal{R}(T)}=Y$ the {\sl spectral cut-off} method is appropriate due to
its unlimited qualification which avoids saturation of the method (cf.,~e.g.,~\cite[Example~4]{Mathe04}). However, the use of spectral cut-off requires the knowledge of the singular system  $\{\sigma_i(T),u_i(T),v_i(T)\}_{i=1}^\infty$ of the compact operator $T$, at least for $i=1,2,...,N$, up to some stopping index $N$, which plays the role of a regularization parameter and occurs in case of noisy data $y^\delta \in Y$ obeying the noise model $\|y-y^\delta\|\le \delta$  with noise level $\delta>0$ inside the regularization procedure
$$x_N^\delta := \sum \limits_{i=1}^N \frac{1}{\sigma_i(T)}\langle y^\delta,v_i(T)\rangle u_i(T) $$
(cf.,~e.g.,~\cite[p.36]{Kirsch11}).

In this note, we restrict our considerations with respect to equation (\ref{eq:opeq}) to the family of {\sl Riemann-Liouville fractional integral operators} $T:=J^\alpha$ defined for all exponents $0<\alpha<\infty$ as {\sl compact} operators
\begin{equation} \label{eq:family}
[J^\alpha x](s):=\int \limits_0^s \frac{(s-t)^{\alpha-1}}{\Gamma(\alpha)} \,x(t) \,dt, \qquad 0 \le s \le 1,
\end{equation}
mapping in the separable infinite dimensional Hilbert space $X=Y:=L^2([0,1])$ of quadratic integrable Lebesgue-measurable real functions over the unit interval $[0,1]$,
which are of particular interest in the mathematical literature. Namely the linear operator equation
\begin{equation} \label{eq:opeq}
J^\alpha x =y, \qquad y \in \mathcal{R}(J^\alpha),
\end{equation}
is solved in a unique manner by the simple formula
$$x= D^\alpha y$$
whenever the right-hand side $y \in L^2([0,1])$ belongs to the range $\mathcal{R}(J^\alpha)$ of the operator $J^\alpha$.
Here, $D^\alpha$ designates the $\alpha$-fold fractional derivative of the function $y$.

The following facts are well-known from the literature (see, for example, \cite{Engl97,GoVe91,Kress89}):

\begin{fact} \label{fact:1}
For all real numbers $0 <\alpha <\infty$ the linear convolution operators $J^\alpha$ mapping in $L^2([0,1])$ are {\sl injective} and {\sl compact},
and so are the adjoint operators
\begin{equation} \label{eq:adjoint}
[(J^\alpha)^* z](t):=\int \limits_t^1 \frac{(s-t)^{\alpha-1}}{\Gamma(\alpha)} z(s)\, ds, \qquad 0 \le t \le 1,
\end{equation}
too. Hence, the range
$\mathcal{R}(J^\alpha)$ is a {\sl dense} and {\sl non-closed} subset of  $L^2([0,1])$. Consequently, the operator equation (\ref{eq:opeq}) is {\sl ill-posed of type~II} in the sense of Nashed \cite{Nashed86}.
The linear Volterra integral operators $J^\alpha$ are linear Fredholm integral operators with quadratically integrable kernel and hence Hilbert-Schmidt operators
whenever $\frac{1}{2}<\alpha<\infty$.
\end{fact}

\begin{fact} \label{fact:2}
For every $0<\alpha<\infty$ the operator $J^\alpha$ possesses a uniquely determined singular system $\{\sigma_i(J^\alpha),u_i(J^\alpha),v_i(J^\alpha)\}_{i=1}^\infty$ with ordered singular values $\sigma_1(J^\alpha) > \sigma_2(J^\alpha) > ... >0$, where
$\lim_{i \to \infty} \sigma_i(J^\alpha)=0$, and two orthonormal systems $\{u_i(J^\alpha)\}_{i=1}^\infty $ and $\{v_i(J^\alpha)\}_{i=1}^\infty$, which are both complete in $L^2([0,1])$, such that for $i=1,2,...$
\begin{equation} \label{eq:SVD}
J^\alpha u_i(J^\alpha) = \sigma_i(J^\alpha)\, v_i(J^\alpha) \quad \mbox{and} \quad (J^\alpha)^* v_i(J^\alpha) = \sigma_i(J^\alpha)\, u_i(J^\alpha)\,.
\end{equation}
\end{fact}

The focus of the present note is on the case
$\alpha=n$ with natural numbers $n=1,2,...$, where $D^n y=y^{(n)}$ coincides with the $n$-fold generalized derivative of the Sobolev space function $y \in \mathcal{R}(J^n) \subset H^n([0,1])$. We try to
answer the frequently asked question why only for $n=1$ the complete {\sl singular system} $\{\sigma_i(J^n),u_i(J^n),v_i(J^n)\}_{i=1}^\infty$  of $J^n$ is made explicit in many textbooks and
papers, but for $n \ge 2$ such a detailed discussion is mostly avoided.
One of the reasons may be that there are no nice explicit formulas for the
singular systems.  We describe them by implicit transcendental equations which
become increasingly unhandy as $n$ grows.  Therefore we employ symbolic
computation (in form of the computer algebra system Mathematica) to derive
some of the formulas presented here. Also the arbitrary-precision arithmetic
that is available through such a system is crucial for finding some of the
numerical approximations.

\section{Singular value asymptotics of Riemann-Liouville fractional integral operators} \label{se2}

Many authors have discussed upper and lower bounds for the singular values $\sigma_i(J^\alpha)$ of $J^\alpha$ aimed at deriving a singular value asymptotics with respect to the Riemann-Liouville fractional integral operators
mapping in $L^2([0,1])$ in the case of specific exponents $\alpha$ and exponent intervals. However, in the paper \cite{VuGor94} we find the complete
asymptotics:

\begin{pro} \label{pro:asym}
For all $0<\alpha<\infty$ there exist constants $$0<\underline c (\alpha) \le \overline c (\alpha) < \infty$$ such that
$$\underline c (\alpha)  \,i^{-\alpha} \le  \sigma_i(J^\alpha)   \le    \overline c (\alpha) \,i^{-\alpha}.  $$
\end{pro}

As a consequence of Proposition~\ref{pro:asym} the {\sl degree of ill-posedness} (cf.~\cite{HoWo05}) of the operator equation (\ref{eq:opeq}) is $\alpha$ and grows with the level of integration. Abel integral equations
($0<\alpha<1$) are {\sl weakly ill-posed} and the problem of $n$-fold differentiation with $\alpha=n \in \IN$  and
\begin{equation} \label{eq:nfold}
\underline c (n)  \,i^{-n} \le  \sigma_i(J^n)   \le    \overline c (n) \,i^{-n}, \qquad i=1,2,...
\end{equation}
is {\sl mildly ill-posed}. No {\sl severely (exponentially) ill-posed} problem occurs in the context of equation
(\ref{eq:opeq}).

\section{The boundary value problem for the singular value decomposition}
\label{se3}

By deriving $u_i(J^n)\;(i \in \IN)$ from the well-known equations
\begin{alignat*}{2}
  [J^n u_i(J^n)]^{(n)} &= u_i(J^n) &&= \sigma_i(J^n)[v_i(J^n)]^{(n)}, \\
  [(J^n)^*v_i(J^n)]^{(n)} &= (-1)^n v_i(J^n) &&= \sigma_i(J^n) [u_i(J^n)]^{(n)},
\end{alignat*}
one can verify the singular system of $J^n$ from the following boundary value problem of an ordinary differential equation of order~$2n$:
\begin{equation} \label{eq:nfoldode}
\left\{ \begin{array}{l} \lambda\,u^{(2n)}(t)+(-1)^{(n+1)}u(t) =0, \qquad 0<t<1\,,\\ u(1)=u^\prime(1)=...=u^{(n-1)}(1)=0\,,\\u^{(n)}(0)=u^{(n+1)}(0)=...= u^{(2n-1)}(0) =0\,. \end{array}\right.
\end{equation}
More precisely, we are searching for all eigenvalues $\lambda>0$ such that the system (\ref{eq:nfoldode}) possesses nontrivial solutions $0 \not=u \in L^2([0,1])$. According to Proposition~\ref{pro:asym}, there will be with $\lambda=\lambda_i$ an infinite sequence $\lambda_1 > \lambda_2 > ...>0$ of ordered eigenvalues  $\lambda_i:=(\sigma_i(J^n))^2$ which are the $i$-largest eigenvalues of both operators $(J^n)^*J^n$ and  $J^n(J^n)^*$.   Moreover, with $u:=u_i(J^n)$ there will be an associated orthonormal eigensystem $\{u_i(J^n)\}_{i=1}^\infty $ which leads with $v_i(J^n):=\frac{1}{\sigma_i(J^n)}J^n u_i(J^n)$ to the  orthonormal eigensystem $\{v_i(J^n)\}_{i=1}^\infty$. Thus the singular system  $\{\sigma_i(J^n),u_i(J^n),v_i(J^n)\}_{i=1}^\infty$ of $J^n$ is complete.


The computation of the eigensystem follows a schema listed in the algorithm below that has been frequently used in the literature for the case $n=1$ (for results see Section~\ref{sec:4}) and can be applied to any larger integer $n \in \mathbb{N}$.
This approach is based on the zeros $\nu$ of the characteristic polynomial
	\begin{equation}\label{char_pol}
		p_n(\nu) = \lambda \nu^{2n}+(-1)^{n+1}
	\end{equation}
of the homogeneous differential equation $\lambda\,u^{(2n)}(t)+(-1)^{(n+1)}u(t) =0 $ of order~$2n$ occurring in the boundary value problem \eqref{eq:nfoldode}.
It is clear that these zeros $\nu$ obey the equation
	\begin{equation} \label{eq:findzero}
		\lambda \nu^{2n}=(-1)^{n}.
	\end{equation}
For fixed $\lambda>0$, the solutions~$u$ of this ODE are characterized by a corresponding fundamental system
  $$\bigl\{\varphi_k(t) \mathrel{\big|} k=0,\dots, 2n-1\bigr\}$$
such that
\begin{equation} \label{eq:u}
	u(t) = \sum_{k=0}^{2n-1} \gamma_k \varphi_{k}(t).
\end{equation}
Each non-zero coefficient vector $\gamma=(\gamma_0,\gamma_1, \dots ,\gamma_{2n-1})^{T}\in \mathbb{R}^{2n}$ represents a non-trivial solution of the ODE. Taking into account the required initial and terminal conditions
of the boundary value problem~\eqref{eq:nfoldode}, the vector $\gamma$ must satisfy the linear system $A_n(\lambda) \gamma =0$ with a singular
quadratic matrix
\begin{equation}\label{eq:matrix_An}
A_n(\lambda)=
	\begin{pmatrix}
		\varphi_0(1)&\varphi_1(1) & \dots & \varphi_{2n-1}(1)\\
		\varphi_0'(1)&\varphi_1'(1) & \dots & \varphi_{2n-1}'(1)\\
		\vdots & \vdots &  & \vdots \\
		\varphi_0^{(n-1)}(1)&\varphi_1^{(n-1)}(1) & \dots & \varphi_{2n-1}^{(n-1)}(1)\\
		\varphi_0^{(n)}(0)&\varphi_1^{(n)}(0) & \dots & \varphi_{2n-1}^{(n)}(0)\\
		\vdots & \vdots & & \vdots \\
		\varphi_0^{(2n-1)}(0)&\varphi_1^{(2n-1)}(0) & \dots & \varphi_{2n-1}^{(2n-1)}(0)\\
	\end{pmatrix} \in \mathbb{R}^{2n \times 2n}.
\end{equation}
This means that only those $\lambda>0$ for which
\begin{equation} \label{eq:eigen}
	\det  (A_n(\lambda)) = 0
\end{equation}
yield non-zero vectors $\gamma$ such that $u$ in \eqref{eq:u} is non-trivial. It can be seen that only a countable set $\{\lambda_i\}_{i=1}^\infty$ of values $\lambda>0$ satisfies \eqref{eq:eigen}. This set consists of the eigenvalues of $J^{\ast}_{n}J_{n}$ with associated eigenfunctions
$$u_i(t) = \sum_{k=0}^{2n-1} \gamma_k^{(i)} \varphi_{i,k}(t),$$
where $\{\varphi_{i,0},\dots,\varphi_{i,2n-1}\}$ is the fundamental system associated to the eigenvalue~$\lambda_i$, and the vector $\gamma^{(i)}=\bigl(\gamma^{(i)}_0,\gamma^{(i)}_1, \dots ,\gamma^{(i)}_{2n-1}\bigr)^{T}\in \mathbb{R}^{2n}$ satisfies the linear system $A_n(\lambda_i) \gamma^{(i)}=0$ and is normalized such that $\|u_i\|_{L^2([0,1])}=1$.

Now we are ready to formulate the algorithm for obtaining the desired eigenvalues and eigensystems.

\begin{algorithm} \label{alg:algo}
\begin{itemize}
\item[]
\item[(i)] Compute, by solving equation \eqref{eq:findzero}, the $2n$  zeros of the characteristic polynomial $p_n(\nu)$ of the ODE occurring in problem \eqref{eq:nfoldode}.
\item[(ii)]  Construct the fundamental system $\bigl\{\varphi_k(t) \mathrel{\big|} k=0,\dots, 2n-1 \bigr\}$ of the ODE for arbitrary $\lambda>0$.
\item[(iii)] Form the $(2n\times 2n)$-matrix $A_n(\lambda)$ that expresses the initial and terminal conditions occurring in \eqref{eq:nfoldode}.
\item[(iv)] Determine the eigenvalues $\lambda_i,\,i=1,2,...,$ of $J^{\ast}_{n}J_{n}$ by solving the equation  $\det (A_n(\lambda_i)) = 0$.
\item[(v)] Calculate the eigenfunctions $u_i=\sum_{k=0}^{2n-1}  \gamma_k^{(i)} \varphi_{i,k}(t)$ for all $i=1,2,...,$ such that  $\|u_i\|_{L^2(0,1)}=1$.
\end{itemize}
\end{algorithm}
Although this algorithm seems to be straightforward, we will see that only steps (i)-(iii) can be done explicitly.
\begin{pro}
The zeros of the characteristic polynomial \eqref{char_pol} are given by
\begin{equation}
	\nu_k= \left (\frac{1}{\lambda} \right )^{\!\frac{1}{2n}}{\exp}{\left(i\,\frac{[n]_2\pi+2\pi  k}{2n}\right)}, \hspace{1cm} k=0,\dots, 2n-1,
\end{equation}
where $[n]_2:= n \mathrel{\mathrm{mod}} 2$ for $n\in \mathbb{N}$.
\end{pro}
\begin{proof}

	For $n$ even, we have to solve the equation $\nu^{2n}=1/\lambda$. Its zeros are given as
	\begin{equation}
		\nu_k = \left (\frac{1}{\lambda} \right )^{\!\frac{1}{2n}}e^{i\frac{2\pi  k}{2n}}, \hspace{1cm} k=0,\dots, 2n-1.
	\end{equation}
	For $n$ odd, the equation reads
	\begin{equation}
			\nu^{2n}=\frac{-1}{\lambda}=\frac{e^{i\pi}}{\lambda}
	\end{equation}
having the roots
\begin{equation}
	\nu_k = \left (\frac{1}{\lambda} \right )^{\!\frac{1}{2n}}e^{i\frac{\pi+2\pi  k}{2n}}, \hspace{1cm} k=0,\dots, 2n-1.\qedhere
\end{equation}
\end{proof}
As roots of the characteristic polynomial, the $\nu_k$ are either real, or if one root is complex, then its conjugate is also a root. In detail, we have the assertions of the following proposition.

\begin{pro}\label{rootsandconjugates}
For $n$ even, there exist two real zeros
\begin{eqnarray}
\nu_0 &=& \lambda^{-\frac{1}{2n}}\\
\nu_n	&=& -\lambda^{-\frac{1}{2n}}.
\end{eqnarray}
Additionally, we have
\begin{equation}
	\overline{\nu_k} = \nu_{2n-k}, \hspace{1cm}k=1,\dots, n-1.
\end{equation}
If $n$ is odd, all zeros are complex and satisfy
\begin{equation}
	\overline{\nu_k} = \nu_{2n-(k+1)}, \hspace{1cm}k=0,\dots ,n-1.
\end{equation}
\end{pro}

\begin{proof}
	If $n$ is even, the values of $\nu_k$ for $k=0$ and $k=n$ are evident. Moreover, we have for $k=1,\dots, n-1$
	\begin{equation}
		\overline{\nu_k}= \lambda^{-\frac{1}{2n}} \overline{e^{^{i\frac{\pi k}{n}}}}= \lambda^{-\frac{1}{2n}} e^{^{-i\frac{\pi k}{n}}}= \lambda^{-\frac{1}{2n}} e^{^{i\pi\frac{2n-k}{n}}}=\nu_{2n-k}.
	\end{equation}
	In the case that $n$ is odd, the zeros are given by
	\begin{equation}
		\nu_k= \lambda^{-\frac{1}{2n}}e^{^{\frac{i\pi}{2n}(1+2k)}},
	\end{equation}
	and as there is no $k\in \mathbb{N}$ such that $1+2k$ equals zero or a multiple of $2n$, there are no real roots. Additionally, we have
	\begin{align*}
		\lambda^{\frac{1}{2n}}\overline{\nu_k} &= e^{^{\frac{i\pi}{2n}(1+2k)}}= e^{^{-\frac{i\pi}{2n}(1+2k)+2\pi i}}= e^{^{\frac{i\pi}{n}(2n-k-\frac{1}{2})}}\\
		&= e^{^{\frac{i\pi}{n}(\frac{1}{2}+(2n-(k+1))}}=\lambda^{\frac{1}{2n}}\nu_{2n-(k+1)}. \qedhere
	\end{align*}
\end{proof}
It is well known that a complex root and its complex conjugate create a pair of real fundamental solutions of an ODE. Specifically, a complex root  $\nu_k=\alpha_k \pm i\beta_k$ with multiplicity one creates the two real fundamental solutions
\begin{equation}\label{FS}
	e^{\alpha_k\cdot t}\cos(\beta_k\cdot t),\hspace{1cm}e^{\alpha_k\cdot t}\sin(\beta_k\cdot t).
\end{equation}
For what follows, let us denote the roots of the characteristic polynomial 	$p_n(\lambda)$ by $\nu_k^{(e)}$ if $n$ is even and by  $\nu_k^{(o)}$ if $n$ is odd. Then
we obtain the following result.
\begin{pro}
	Let $\nu_k^{(e,o)}=\alpha_k^{(e,o)} + i\beta_k^{(e,o)}$ denote the roots of the characteristic polynomial \eqref{char_pol} and
\begin{equation}\label{FundamentalS}
			\left \{\varphi_0,~\varphi_1 , \dots , \varphi_{2n-1}\right \}
		\end{equation}
the fundamental system of the ODE occurring in \eqref{eq:nfoldode}. Then we have:
	\begin{enumerate}
		\item[(a)] If $n$ is even then, for $k=1,\dots , n-1$, the system (\ref{FundamentalS}) is characterized as
				\begin{eqnarray*}
			\varphi_0(t)&=& e^{(\lambda^{^{-1/2n}})\cdot t}\\
			\varphi_1(t)&=& e^{(-\lambda^{^{-1/2n}})\cdot t}\\
			\varphi_{2k}(t)&=& e^{\alpha_k^{(e)}\cdot t}\cos(\beta_k^{(e)}\cdot t)\\
			\varphi_{2k+1}(t)&=& e^{\alpha_k^{(e)}\cdot t}\sin(\beta_k^{(e)}\cdot t).
		\end{eqnarray*}
				\item[(b)] If $n$ is odd then, for $k=1,\dots , n-1$, the system (\ref{FundamentalS}) is characterized as
		\begin{eqnarray*}
			\varphi_{2k}(t)&=& e^{\alpha_k^{(o)}\cdot t}\cos(\beta_k^{(o)}\cdot t)\\
			\varphi_{2k+1}(t)&=& e^{\alpha_k^{(o)}\cdot t}\sin(\beta_k^{(o)}\cdot t).
		\end{eqnarray*}
			\end{enumerate}
\end{pro}
\begin{proof}
	Taking into account  \eqref{FS}, the proof follows by the characterization of the roots and their complex conjugates in Proposition~\ref{rootsandconjugates}.
\end{proof}


\section{The onefold integration operator} \label{sec:4}
Along the lines outlined above in the algorithm and frequently presented in the literature one finds the singular system for $n=1$, i.e.~for the simple integration operator $J^1$, from the ODE system
\begin{equation*} \label{eq:1foldode}
\left\{ \begin{array}{l} \lambda\,u^{\prime\prime}(t)+u(t) =0, \qquad 0<t<1\,,\\ u(1)=0\,,\\u^\prime(0)=0\,. \end{array}\right.
\end{equation*}
The explicit structure of this singular system is outlined in the following proposition.
\begin{pro} \label{pro:simple}
For $n=1$ we have the explicitly given singular system
\[
  \left\{
  \sigma_i=\tfrac{2}{(2i-1)\pi},
  u_i(t) = \sqrt{2}\,{\cos}{\left({\left(i-\tfrac12\right)} \pi t\right)},
  v_i(t) = \sqrt{2}\,{\sin}{\left({\left(i-\tfrac12\right)} \pi t\right)}
  \right\}_{i=1}^\infty
\]
of the operator $J^1$ mapping in the Hilbert space $L^2([0,1])$. Hence, formula (\ref{eq:nfold}) applies in the form
$$ \frac{1}{\pi}  \,i^{-1} \le  \sigma_i(J^1)   \le \frac{2}{\pi} \,   i^{-1}, \qquad i=1,2,...\, .$$
\end{pro}

\section{The twofold integration operator} \label{se4}
\subsection{General assertions}
In the case $n=2$, i.e.~for the twofold integration operator $J^2$, the ODE-system (\ref{eq:nfoldode}) attains the form:
\begin{equation} \label{eq:twofoldode}
\left\{ \begin{array}{l} \lambda\,u^{(4)}(t)-u(t) =0, \qquad 0<t<1\,,\\ u(1)=u^\prime(1)=0\,,\\u^{\prime \prime}(0)=u^{\prime \prime \prime}(0)=0\,. \end{array}\right.
\end{equation}
For the eigenfunctions $u=u_i(J^2)$ it is a necessary condition that they satisfy the homogeneous fourth-order differential equation in (\ref{eq:twofoldode}), which implies the ansatz structure via the corresponding fundamental system as
\[
  u(t)=\gamma_1\, {\exp}{\left(\frac{t}{\lambda^{1/4}}\right)} +
  \gamma_2\, {\exp}{\left(-\frac{t}{\lambda^{1/4}}\right)} +
  \gamma_3\, {\sin}{\left(\frac{t}{\lambda^{1/4}} \right)} +
  \gamma_4\, {\cos}{\left(\frac{t}{\lambda^{1/4}} \right)}.
\]
To obtain such $u \not=0$, the linear $(4 \times 4)$-system of equations
\[
  A_2(\mu)\cdot(\gamma_1,\gamma_2,\gamma_3,\gamma_4)^T=(0,0,0,0)^T
  \quad\text{with}\quad \mu:=\lambda^{-1/4}
\]
must have a singular matrix $A_2(\mu)$, which means that
$${\rm det}\left(\begin{array}{rrrr} e^\mu & e^{-\mu} & \sin(\mu) & \cos(\mu)\\e^\mu & -e^{-\mu} & \cos(\mu) & -\sin(\mu) \\ 1 & 1 & 0 & -1\\ 1 & -1 & -1 & 0\end{array} \right)=4\bigl(\cos(\mu)\cosh(\mu)+1\bigr) =0 \,.$$
This leads to the following proposition:

\begin{pro} \label{pro:lambda}
The eigenvalues $\lambda$ of the operator $(J^2)^*J^2$ are the solutions of the nonlinear transcendental equation
\begin{equation} \label{eq:lambda}
\cos\left( \frac{1}{\lambda^{1/4}}\right)\cdot \cosh\left( \frac{1}{\lambda^{1/4}}\right)+1 =0.
\end{equation}
\end{pro}

\medskip

In the next subsection we motivate the fact that the sequence
$\{\lambda_i\}_{i=1}^\infty$ of solutions to (\ref{eq:lambda}) is of the form
$\lambda_i=\frac{1}{((i-\frac{1}{2})\pi+\varepsilon_i)^4}\;(i=1,2,...)$, where
the sequence $\{\varepsilon_i\}_{i=1}^\infty$ tends to zero exponentially
fast. This gives evidence that the singular values $\sigma_i(J^2)$ are very
close to
\[
  \frac{1}{{\left((i+\frac{1}{2})\pi\right)}{}^2}\quad(i=1,2,...)
\]
for sufficiently large $i$, which is in accordance with the assertion of
Proposition~\ref{pro:asym} in the case $\alpha:=2$.

\subsection{On the zeros of the function $f(z)=\cos(z) \cosh(z)+1 =0$ } \label{se5}

To our knowledge there does not exist a closed form for the zeros of the
transcendental equation $f(z) = \cos(z) \cosh(z) + 1 = 0$. From the form of
the equation it becomes apparent that there are infinitely many zeros, whose
distribution is approximately $\pi$-periodic. Applying Newton's method to
$f(z)$, we find the following numeric values for the first few positive roots:
\begin{align*}
  z_1 &= 1.875104068711961166445308241078214...\\ 
  z_2 &= 4.694091132974174576436391778019812...\\ 
  z_3 &= 7.854757438237612564861008582764570...\\ 
  z_4 &= 10.99554073487546699066734910785470...\\ 
  z_5 &= 14.13716839104647058091704681255177...\\ 
\end{align*}
The almost-periodic behavior of $\{z_i\}_{i=1}^\infty$ suggests to write
\begin{equation}\label{eq.zn}
  z_i = (i - \tfrac12)\pi + \eps_i,
\end{equation}
where the sequence $\{\eps_i\}_{i=1}^\infty$ tends to zero exponentially
fast. Our first goal is to derive a bound on the absolute value of~$\eps_i$,
thereby proving the claimed asymptotic behavior of $\{\eps_i\}_{i=1}^\infty$.

For this purpose, consider the function
\[
  g(z):=f(z)-1=\cos(z)\cosh(z),
\]
whose zeros are at the positions
\[
  \zeta_i:=(i-\tfrac12)\pi \quad(i=1,2,...).
\]
The locations where the graph of $g(z)$ intersects the line $y=-1$ are exactly
the zeros of~$f(z)$. From $g''(z)=-2\sin(z)\sinh(z)$ we see that $g(z)$ is
convex if $\sin(z)<0$ and that $g(z)$ is concave when $\sin(z)>0$.

\begin{lemma}\label{lemma.e1}
If $i\geq1$ is an odd integer, then $0<\eps_i<2\,\e^{-\zeta_i}$.
\end{lemma}
\begin{proof}
If $i$ is odd, then $g'(\zeta_i)=-\cosh(\zeta_i)$. Hence $g$ has a negative
slope at $\zeta_i$ and therefore $\eps_i>0$. Since $g$ is concave in the interval
$\bigl(\zeta_i-\frac{\pi}{2},\zeta_i+\frac{\pi}{2}\bigr)$, it follows that in
this interval the tangent to~$g$ at $\zeta_i$ is above~$g$. This tangent
intersects the line $y=-1$ at $\zeta_i+1/\cosh(\zeta_i)$, which yields the
desired upper bound on~$\eps_i$:
\[
  \eps_i < \frac{1}{\cosh(\zeta_i)} = \frac{2}{\e^{\zeta_i} + \e^{-\zeta_i}} < \frac{2}{\e^{\zeta_i}}.
  \qedhere
\]
\end{proof}

\begin{lemma}\label{lemma.e2}
If $i\geq2$ is an even integer, then $0<-\eps_i<4\,\e^{-\zeta_i}$.
\end{lemma}
\begin{proof}
In this case, $g'(\zeta_i)=\cosh(\zeta_i)$. Hence $g$ is increasing, which
means that its graph intersects $y=-1$ left to $\zeta_i$, thus $\eps_i<0$.
Moreover, $g$ is convex in the interval
$\bigl(\zeta_i-\frac{\pi}{2},\zeta_i+\frac{\pi}{2}\bigr)$, and therefore the
tangent to~$g$ at $\zeta_i$ is below~$g$. Unfortunately, the intersection
between this tangent and $y=-1$ does not deliver an upper bound on~$|\eps_i|$.

Instead, we define $\xi_i := \zeta_i-4\,\e^{-\zeta_i}$ and show that $g(\xi_i)<-1$,
yielding the claimed bound on~$\eps_i$. Equivalently, we show $-g(\xi_i)>1$:
\begin{align*}
-g(\xi_i) &= \sin\bigl(4\,\e^{-\zeta_i}\bigr) \cosh(\xi_i) \\
 &= \tfrac12 \sin\bigl(4\,\e^{-\zeta_i}\bigr) \bigl(\e^{\xi_i}+\e^{-\xi_i}\bigr) \\
 &> \tfrac12 \sin\bigl(4\,\e^{-\zeta_i}\bigr) \, \e^{\xi_i} \\
 &= \tfrac12 \sin\bigl(4\,\e^{-\zeta_i}\bigr) \, \e^{\zeta_i} \exp\bigl(-4\,\e^{-\zeta_i}\bigr) \\
 &= \frac{\sin(4x)}{2x} \, \e^{-4x} \quad\text{with } x = \e^{-\zeta_i}.
\end{align*}
The function $h(x)= \bigl(\sin(4x)/(2x)\bigr) \e^{-4x}$ is monotonically
decreasing in the interval $\bigl(0,\e^{-\zeta_1}\bigr)$ with $h(0)=2$ and
$h\bigl(\e^{-\zeta_2}\bigr)\approx 1.92899$. In particular, we have $h(x)>1$
in this interval and therefore $h(\e^{-\zeta_i})>1$ for all~$i\geq2$, which
implies our claim on $g(\xi_i)$. 
\end{proof}

\begin{remark}
The factor~$4$ in the previous lemma is because of our crude estimate;
actually we have for even and odd~$i$
\[
  |\eps_i| \sim 2\,\e^{-\zeta_i}\quad\text{for } i\to\infty.
\]
\end{remark}

\begin{remark}
Analogous statements can be made about the negative roots of the
function~$f(z)$; they follow immediately by symmetry since $f$ is an even
function.
\end{remark}

Instead of a bound on~$\eps_i$, we can also derive an exact expression for it in
the form of an infinite series.  Plugging the representation~\eqref{eq.zn}
into the equation $f(z_i)=0$, one obtains
\[
  (-1)^i \sin(\eps_i) \cosh\bigl(\bigl(i - \tfrac12\bigr)\pi + \eps_i\bigr) = -1,
\]
or equivalently
\[
  \sin(\eps_i) + (-1)^i \frac{2}{w_i^{-1}+w_i} = 0
  \qquad\text{with }
  w_i = \exp\bigl(-\bigl(i - \tfrac12\bigr)\pi - \eps_i\bigr).
\]
We expand the left-hand side as a geometric series in~$w_i^2$, which gives
\[
  \sin(\eps_i) + 2\, (-1)^i \sum_{k=0}^\infty (-1)^k w_i^{2k+1}.
\]
Next, we write $(-1)^i w_i=x_i\cdot\exp(-\eps_i)$ with
$x_i=(-1)^i\exp\bigl(-\bigl(i + \tfrac12\bigr)\pi\bigr)$,
and perform Taylor expansion with respect to~$\eps_i$:
\[
  \sum_{j=0}^\infty \frac{(-1)^j}{(2j+1)!} \eps_i^{2j+1} +
  2 \sum_{k=0}^\infty (-1)^k x_i^{2k+1} \sum_{j=0}^\infty \frac{(-2k-1)^j}{j!} \eps_i^j.
\]
Formally speaking, this is a bivariate power series in the variables~$x_i$ and~$\eps_i$.
Making an ansatz for~$\eps_i$, i.e., substituting for $\eps_i$ a power series in~$x_i$ with
undetermined coefficients,
\[
  \eps_i = \sum_{i=1}^\infty a_i x_i^i,
\]
we obtain a univariate series:
\begin{multline*}
  (a_1 + 2) x_i
  + (a_2 - 2a_1) x_i^2
  + \left(a_3 - 2a_2 - \tfrac{1}{6}a_1^3 + a_1^2 - 2\right) x_i^3 \\
  + \left(a_4 - 2a_3 - \tfrac{1}{2}a_1^2a_2 + 2a_1a_2 - \tfrac{1}{3}a_1^3 + 6 a_1\right) x_i^4
  + \dots
\end{multline*}
Coefficient comparison with respect to~$x_i$ then allows us to compute the
unknown coefficients~$a_i$; note that in the coefficient of~$x_i^k$ the
indeterminate~$a_k$ appears linearly and can therefore be easily computed
from the previous ones:
\[
  a_1 = -2,\; a_2 = -4,\; a_3 = -\tfrac{34}{3},\; a_4 = -\tfrac{112}{3},\;
  a_5 = -\tfrac{2006}{15},\; a_6 = -\tfrac{1516}{3},\; \dots
\]

\begin{remark}
We have computed the first 100 coefficients $a_i$ symbolically, but we were
not able to identify a nice closed form for them. They do not satisfy a (nice)
linear recurrence equation with polynomial coefficients, either.  Also in the
OEIS~\cite{Sloane}, we could not find any information about these numbers.
\end{remark}


\subsection{Eigenfunctions}

With the acquired knowledge on the eigenvalues $\lambda$ of the operator $(J^2)^*J^2$,
we are able to derive the corresponding eigenfunctions $u_i$, at least numerically.
Recall the fundamental system
\[
  \gamma_1\, {\exp}{\left(\frac{t}{\lambda^{1/4}}\right)} +
  \gamma_2\, {\exp}{\left(-\frac{t}{\lambda^{1/4}}\right)} +
  \gamma_3\, {\sin}{\left(\frac{t}{\lambda^{1/4}} \right)} +
  \gamma_4\, {\cos}{\left(\frac{t}{\lambda^{1/4}} \right)}.
\]
By plugging the computed values for $\lambda_i$, $1\leq i\leq5$, into the
matrix $A_2(\lambda)$, we can determine the constants
$\gamma_1,\gamma_2,\gamma_3,\gamma_4$.  The results are shown in
Table~\ref{tab:gammas2} (after the normalization $\|u_i\|_{L^2([0,1])}=1$) and
the eigenfunctions themselves are plotted in Figure~\ref{fig:efs2}.

\begin{table}
\[
\begin{array}{l|rrrrr}
  i & \hfill\lambda_i\hfill & \hfill\gamma_1\hfill &
  \hfill\gamma_2\hfill & \hfill\gamma_3\hfill & \hfill\gamma_4\hfill \\ \hline
  1 & 0.0808907 & 0.1329522 & 0.8670478 & -0.7340955 & 1.0000000\rule{0pt}{12pt} \\
  2 & 0.0020597 & -0.0092337 & 1.0092340 & -1.0184670 & 1.0000000 \\
  3 & 0.0002627 & 0.0003878 & 0.9996122 & -0.9992245 & 1.0000000 \\
  4 & 0.0000684 & -0.0000168 & 1.0000170 & -1.0000340 & 1.0000000 \\
  5 & 0.0000250 & 0.0000007 & 0.9999993 & -0.9999986 & 1.0000000
\end{array}
\]
\caption{}
\label{tab:gammas2}
\end{table}

\begin{figure}
  \begin{center}
    \includegraphics[width=0.8\textwidth]{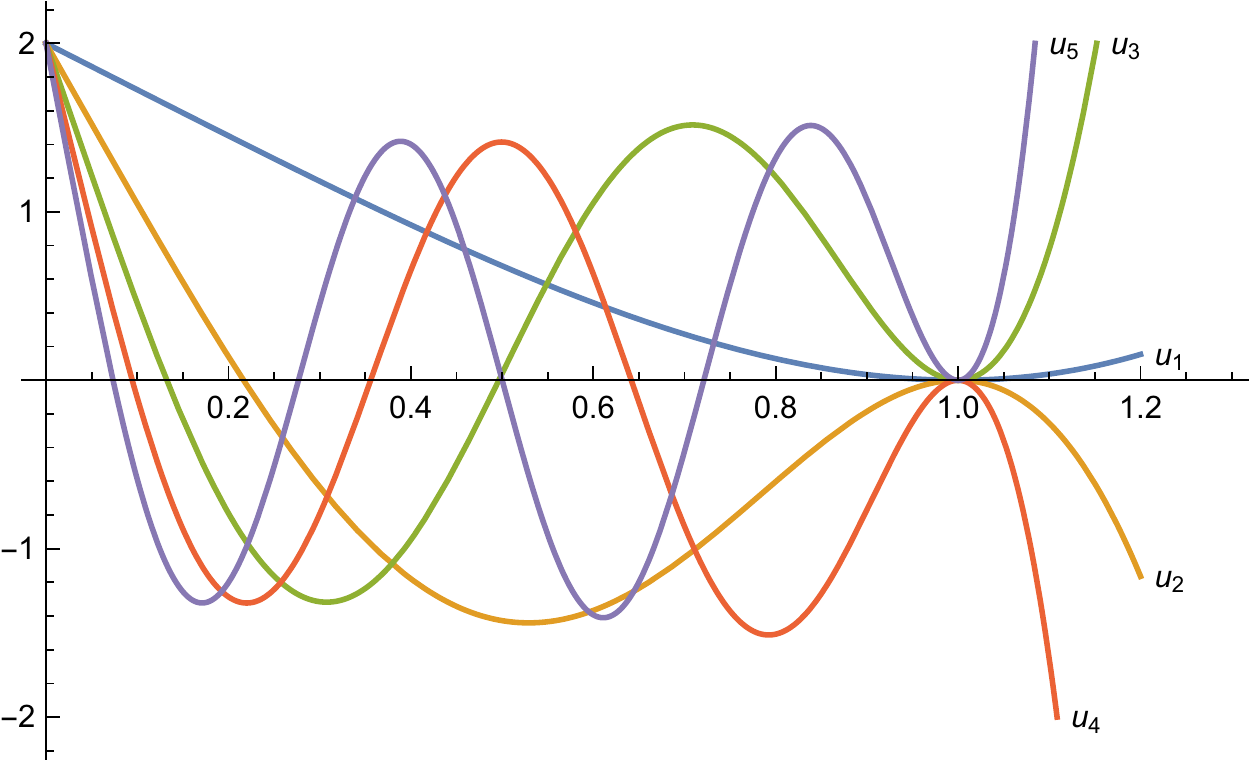}
  \end{center}
  \caption{Eigenfunctions $u_1,\dots,u_5$ for $n=2$}
  \label{fig:efs2}
\end{figure}

\section{The $n$-fold integration operator} \label{se6}

With the notation $\omega_k=\exp\bigl(\frac{i\pi}{2n}(2k+[n]_2)\bigr)$, the
roots $\nu_k$ of the characteristic polynomial~\eqref{char_pol},
$p_n(\nu)=\lambda\nu^{2n}+(-1)^{n+1}$, can be written as
$\nu_k=\lambda^{-\frac{1}{2n}}\omega_k$, $k=0,\dots,2n-1$. Let
$z=\lambda^{-\frac{1}{2n}}$ and write the fundamental system of the ODE
in~\eqref{eq:nfoldode} in terms of the complex exponential functions
$\varphi_k(t)=e^{\omega_kzt}$, then the matrix $A_n$ given
in~\eqref{eq:matrix_An} attains the following form
\[
  A_n(z) = \Bigl(a^{(n)}_{j,k}(z)\Bigr)_{0\leq j,k\leq 2n-1}
  \quad\text{with}\quad
  a^{(n)}_{j,k}(z) = \begin{cases}
    \omega_k^j z^j e^{\omega_kz}, & j<n, \\
    \omega_k^j z^j, & j\geq n.
  \end{cases}
\]
We want to determine the values of $z\neq0$ for which $\det(A_n(z))=0$. Hence
the common factor $z^j$ from the $j$-th row can be removed.  In order to
obtain an explicit expression for the determinant of $A_n(z)$, we first study
the more general matrix
\[
  M_n = \Bigl(m^{(n)}_{j,k}\Bigr)_{0\leq j,k\leq 2n-1}
  \quad\text{with}\quad
  m^{(n)}_{j,k} = \begin{cases} \omega_k^jz_k, & j<n, \\ \omega_k^j, & j\geq n, \end{cases}
\]
where $z_0,\dots,z_{2n-1}$ are indeterminates. The matrix looks as follows:
\[
  M_n = \begin{pmatrix}
    z_0 & z_1 & z_2 & \cdots & z_{2n-1} \\
    \omega_0z_0 & \omega_1z_1 & \omega_2z_2 & \cdots & \omega_{2n-1}z_{2n-1} \\
    \omega_0^2z_0 & \omega_1^2z_1 & \omega_2^2z_2 & \cdots & \omega_{2n-1}^2z_{2n-1} \\
    \vdots & \vdots & \vdots & & \vdots \\
    \omega_0^{n-1}z_0 & \omega_1^{n-1}z_1 & \omega_2^{n-1}z_2 & \cdots & \omega_{2n-1}^{n-1}z_{2n-1} \\
    \omega_0^n & \omega_1^n & \omega_2^n & \cdots & \omega_{2n-1}^n \\
    \vdots & \vdots & \vdots & & \vdots \\
    \omega_0^{2n-1} & \omega_1^{2n-1} & \omega_2^{2n-1} & \cdots & \omega_{2n-1}^{2n-1}
  \end{pmatrix}.
\]
For $z_0=\ldots=z_{2n-1}=1$, the matrix $M_n$ equals the Vandermonde matrix
$V(\omega_0,\dots,\omega_{2n-1})$, which for even $n$ is the Fourier matrix,
since in this case the $\omega_k$ are precisely the $2n$-th complex roots of
unity.

For symbolic $z_0,\dots,z_{2n-1}$, the determinant of $M_n$ is a polynomial in
$z_0,\dots,z_{2n-1}$ that is homogeneous of degree~$n$ and linear in each
variable~$z_k$. Let $I\subset\{0,1,\dots,2n-1\}$ be an index set with
$|I|=n$. We aim at computing the coefficient of the monomial $\prod_{k\in I}
z_k$ in $\det(M_n)$. This corresponds to setting $z_k=0$ for all $k\in
C:=\{0,1,\dots,2n-1\}\setminus I$. By permuting its columns, the matrix~$M_n$
can be transformed into a block matrix of the form
\[
  \left(\begin{array}{c|c}
    V\bigl((\omega_k)_{k\in I}\bigr) \cdot \operatorname{diag}\bigl((z_k)_{k\in I}\bigr) & 0 \\[1ex] \hline
    \ast\rule{0ex}{3ex} & V\bigl((\omega_k)_{k\in C}\bigr) \cdot \operatorname{diag}\bigl((\omega_k^n)_{k\in C}\bigr)
  \end{array}\right).
\]
Moving all columns with index in~$I$ to the first $n$ positions requires
$\sum_{k\in I}k-\frac12n(n-1)$ swaps of neighboring columns. Hence for the
coefficient of the monomial $\prod_{k\in I} z_k$ in $\det(M_n)$ one obtains:
\begin{equation}\label{eq:coeff}
  (-1)^{\sum_{k\in I}k-\frac12n(n-1)} \cdot \det V\bigl((\omega_k)_{k\in I}\bigr)
  \cdot \det V\bigl((\omega_k)_{k\in C}\bigr) \cdot \prod_{k\in C}\omega_k^n.
\end{equation}
From the definition of $\omega_k$ it follows immediately that
$\omega_k^n=(-1)^k$ if $n$ is even, and $\omega_k^n=i\cdot(-1)^k$ if $n$ is odd.
Hence the term $\prod_{k\in C}\omega_k^n$ turns into $i^{[n]_2\cdot n}\cdot(-1)^{\sum_{k\in C}k}$.
The sign in~\eqref{eq:coeff} can now be determined by the parity of
\begin{align*}
  \sum_{k\in I}k-\frac{n(n-1)}{2}+\sum_{k\in C}k
  &= \sum_{k=0}^{2n-1}k-\frac{n(n-1)}{2} \\
  &= \frac{2n(2n-1)}{2}-\frac{n(n-1)}{2} \\
  &= \frac{n(3n-1)}{2} = n(n-1) + \frac{n(n+1)}{2}.
\end{align*}
In addition, by employing the well-known formula for the Vandermonde
determinant, Equation~\eqref{eq:coeff} simplifies to
\begin{equation}\label{eq:cIn}
  (-1)^{n(n+1)/2} \, i^{[n]_2\cdot n} \,
  \Biggl(\prod_{\textstyle\genfrac{}{}{0pt}{}{k,\ell\in I}{k<\ell}}(\omega_\ell - \omega_k)\Biggr)
  \Biggl(\prod_{\textstyle\genfrac{}{}{0pt}{}{k,\ell\in C}{k<\ell}}(\omega_\ell - \omega_k)\Biggr)
  =: c_I^{(n)}.
\end{equation}
Hence the determinant of $M_n$ can be written as follows:
\begin{equation}\label{eq:detMn}
  \det(M_n) = \sum_{\textstyle\genfrac{}{}{0pt}{}{I\subset\{0,\dots,2n-1\}}{|I|=n}} c_I^{(n)}\cdot\prod_{k\in I}z_k.
\end{equation}
Note that for computational purposes the Vandermonde products in the
definition of $c_I^{(n)}$ can always be expressed in terms of $2n$-th roots of
unity, even if $n$ is odd. In this case one has to extract the factor
$i^{n-1}$ which simplifies in combination with the term~$i^{[n]_2\cdot n}$,
namely $i^{n-1}\cdot i^{[n]_2\cdot n}=i$.

\begin{thm}\label{thm:detAn}
Let $\omega_k=\exp\bigl(\frac{i\pi}{2n}(2k+[n]_2)\bigr)$ and $c_I^{(n)}$
as before, and define
$\alpha_I := \operatorname{Re}\bigl(\sum_{k\in I}\omega_k\bigr)$ and
$\beta_I := \operatorname{Im}\bigl(\sum_{k\in I}\omega_k\bigr)$.
Then
\begin{equation}\label{eq:detform}
  \det(A_n(z)) = z^{n(2n-1)} \; \cdot\hspace{-3ex}
  \sum_{\textstyle\genfrac{}{}{0pt}{}{I\subset\{0,\dots,2n-1\}}{|I|=n}} \hspace{-3ex}
  c_I^{(n)} \cosh(\alpha_I z)\cos(\beta_I z).
\end{equation}
\end{thm}
\begin{proof}
We show that $\det(A_n(z))$ is basically a special case of~\eqref{eq:detMn} (up
to the factor $z^{n(2n-1)}$ which comes from the common row factors~$z^j$ that
were omitted in the definition of~$M_n$).  Recall that the variables $z_k$ in
the matrix~$M_n$ have replaced the exponential functions $e^{\omega_kz}$ that
appear in the matrix~$A_n(z)$. Thus, with $z_k=e^{\omega_k z}$ one obtains
\[
  \prod_{k\in I}z_k = {\exp}{\biggl(z\cdot\sum_{k\in I}\omega_k\biggr)} =
  \exp\bigl(z\cdot(\alpha_I+i\beta_I)\bigr).
\]
As before, let $C:=\{0,\dots,2n-1\}\setminus I$ be the set complement of~$I$.
From $\sum_{k\in I} \omega_k + \sum_{k\in C} \omega_k = \sum_{k=0}^{2n-1} \omega_k = 0$,
it follows immediately that
\[
  \alpha_C = -\alpha_I \quad\text{and}\quad \beta_C = -\beta_I.
\]
Moreover, define
\[
  \bar{I} := \{2n-[n]_2-k \mathrel{\mathrm{mod}} 2n \mid k\in I\},
\]
which corresponds to reflecting the set $\{\omega_k\}_{k\in I}$ across the real
axis. Then it follows that
\[
  \alpha_{\bar{I}} = \alpha_I \quad\text{and}\quad \beta_{\bar{I}} = -\beta_I.
\]
Apparently, for those sets~$I$ for which $I=\bar{I}$ holds, one has
$\beta_I=0$.  In such cases, we can combine the exponentials that correspond
to~$I$ and to~$C$, using the identity
\[
  e^{\alpha z} + e^{-\alpha z} = 2\cosh(\alpha z).
\]
Otherwise, if $I\neq\bar{I}$, then $I,\bar{I},C,\bar{C}$ are four pairwise
distinct sets, and their corresponding exponentials can be combined as follows:
\[
  e^{(\alpha+i\beta)z} + e^{(\alpha-i\beta)z} + e^{(-\alpha+i\beta)z} + e^{(-\alpha-i\beta)z} =
  4\cosh(\alpha z)\cos(\beta z).
\]
The asserted formula follows by observing that
$c_I^{(n)}=c_{\bar{I}}^{(n)}=c_{C}^{(n)}=c_{\bar{C}}^{(n)}$ (which is obvious
from symmetry arguments).
\end{proof}

\begin{xmpl}
For $n=2$, we deduce again the transcendental equation satisfied by
$z=\lambda^{-\frac14}$, but now using Theorem~\ref{thm:detAn}:
\begin{itemize}
\item For $I=\{0,2\}$ one has $I=\bar{I}$ and $C=\bar{C}=\{1,3\}$; then it
  follows that
  \begin{align*}
    c_I^{(n)} &= (-1)^3\cdot(\omega_2-\omega_0)\cdot(\omega_3-\omega_1)=-4i, \\
    \alpha_I &= \beta_I = 0.
  \end{align*}
\item For $I=\{0,1\}$ one has $\bar{I}=\{0,3\}$, $C=\{2,3\}$ and
  $\bar{C}=\{1,2\}$; then it follows that
  \begin{align*}
    c_I^{(n)} &= -(\omega_1-\omega_0)(\omega_3-\omega_2)=-(i-1)(-i+1)=-2i,\\
    \alpha_I &= \operatorname{Re}(\omega_0+\omega_1)=1,\\
    \beta_I &= \operatorname{Im}(\omega_0+\omega_1)=1.
  \end{align*}
\end{itemize}
Summing over all possible subsets $I\subset\{0,1,2,3\}$, we get
\begin{multline*}
  2\cdot (-4i)\cosh(0)\cos(0) + 4\cdot(-2i)\cosh(z)\cos(z) = \\
  {-8i}\cdot\bigl(1+\cosh(z)\cos(z)\bigr),
\end{multline*}
which is in accordance with our previous result
(Proposition~\ref{pro:lambda}).
\end{xmpl}

\begin{corollary}\label{cor:eq3}
In the case $n=3$, i.e., for the threefold integration operator $J^3$, the
eigenvalues $\lambda$ of the operator $(J^3)^*J^3$ are the solutions of the
nonlinear transcendental equation
\begin{multline*}
  8\cos\bigl(\lambda^{-\frac16}\bigr) + \cos\bigl(2\lambda^{-\frac16}\bigr)
  +2\cos\bigl(\lambda^{-\frac16}\bigr) \cosh\bigl(\sqrt{3}\lambda^{-\frac16}\bigr) \\
  +16\cos\bigl(\tfrac{1}{2}\lambda^{-\frac16}\bigr)
    \cosh\bigl(\tfrac{\sqrt{3}}{2}\lambda^{-\frac16}\bigr) + 9 = 0.
\end{multline*}
\end{corollary}

\begin{corollary}\label{cor:eq4}
In the case $n=4$, i.e., for the fourfold integration operator $J^4$, the
eigenvalues $\lambda$ of the operator $(J^4)^*J^4$ are the solutions of the
nonlinear transcendental equation (which is obtained from
Theorem~\ref{thm:detAn} after further simplifications):
\begin{align*}
  &\cos\bigl(\sqrt{2} \lambda^{-\frac18}\bigr)
  +\cosh\bigl(\sqrt{2} \lambda^{-\frac18}\bigr)
  +2 \cos\bigl(\lambda^{-\frac18}\bigr) \cosh\bigl(\lambda^{-\frac18}\bigr)
  \\
  &+3 \bigl(\cos\bigl(\sqrt{2} \lambda^{-\frac18}\bigr) + \cosh\bigl(\sqrt{2} \lambda^{-\frac18}\bigr)\bigr)
    \cos\bigl(\lambda^{-\frac18}\bigr) \cosh\bigl(\lambda^{-\frac18}\bigr)
  \\
  &+8 \bigl(\cos\bigl(\lambda^{-\frac18}\bigr)+\cosh\bigl(\lambda^{-\frac18}\bigr)\bigr)
    \cos\bigl(\tfrac{1}{\sqrt{2}}\lambda^{-\frac18}\bigr) \cosh\bigl(\tfrac{1}{\sqrt{2}}\lambda^{-\frac18}\bigr)
  \\
  &+4 \sqrt{2} \sin\bigl(\lambda^{-\frac18}\bigr) \sin\bigl(\tfrac{1}{\sqrt{2}}\lambda^{-\frac18}\bigr)
    \cosh\bigl(\tfrac{1}{\sqrt{2}}\lambda^{-\frac18}\bigr)
  \\
  &-4 \sqrt{2} \cos\bigl(\tfrac{1}{\sqrt{2}}\lambda^{-\frac18}\bigr)
    \sinh\bigl(\lambda^{-\frac18}\bigr) \sinh\bigl(\tfrac{1}{\sqrt{2}}\lambda^{-\frac18}\bigr)
  \\
  &+2 \sqrt{2} \sin\bigl(\lambda^{-\frac18}\bigr) \sin\bigl(\sqrt{2} \lambda^{-\frac18}\bigr)
    \cosh\bigl(\lambda^{-\frac18}\bigr)
  \\
  &-2 \sqrt{2} \cos\bigl(\lambda^{-\frac18}\bigr) \sinh\bigl(\lambda^{-\frac18}\bigr)
    \sinh\bigl(\sqrt{2} \lambda^{-\frac18}\bigr) + 6 = 0.
\end{align*}
\end{corollary}

The transcendental equations given in Corollaries~\ref{cor:eq3}
and~\ref{cor:eq4} can be used to compute accurate approximations to the
eigenvalues~$\lambda_i$, using Newton's method, for example. In
Table~\ref{tab:zeros_3_4} the values for $\lambda_i^{-1/6}$
(resp.~$\lambda_i^{-1/8}$) for $i=1,\dots,5$ are displayed.  As in the
case $n=2$, one observes that the $i$-th value is close to
$(2i-1)\frac{\pi}{2}$ (we give an explanation of this
phenomenon in the next section). The corresponding eigenfunctions~$u_i$
are shown in Figures~\ref{fig:efs_3} and~\ref{fig:efs_4}.

\begin{table}
  \begin{tabular}{l|r@{.}lr@{.}l}
    & \multicolumn{2}{c}{$n=3$ (Cor.~\ref{cor:eq3})}
    & \multicolumn{2}{c}{$n=4$ (Cor.~\ref{cor:eq4})} \\ \hline
    $z_1$ & 2&2247729764011889 & 2&5902718684989891 \rule{0pt}{12pt} \\
    $z_2$ & 4&8026572459190195 & 5&0106222998859963 \\
    $z_3$ & 7&8476475910871745 & 7&8970686069935174 \\
    $z_4$ & 10&9951601546635699 & 10&9949247590502524 \\
    $z_5$ & \quad 14&1371941952108977 & \quad 14&1366518856561214 \rule[-12pt]{0pt}{12pt}
  \end{tabular}
  \caption{Numerical approximations of the first five zeros $z_i$ of the
    transcendental equations in Corollary~\ref{cor:eq3} (with
    $z=\lambda^{-1/6}$) and Corollary~\ref{cor:eq4} (with
    $z=\lambda^{-1/8}$)}
  \label{tab:zeros_3_4}
\end{table}

\begin{figure}
  \begin{center}
    \includegraphics[width=0.8\textwidth]{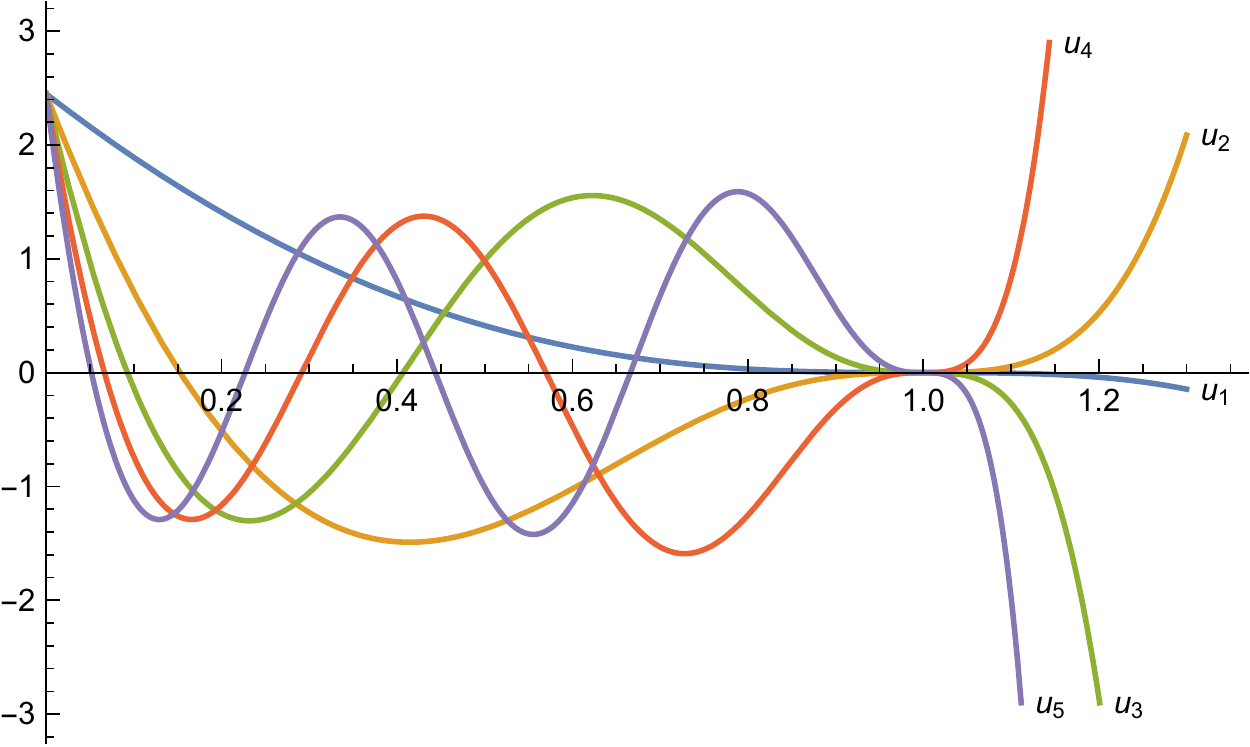}
  \end{center}
  \caption{Eigenfunctions $u_1,\dots,u_5$ for $n=3$}
  \label{fig:efs_3}
\end{figure}

\begin{figure}
  \begin{center}
    \includegraphics[width=0.8\textwidth]{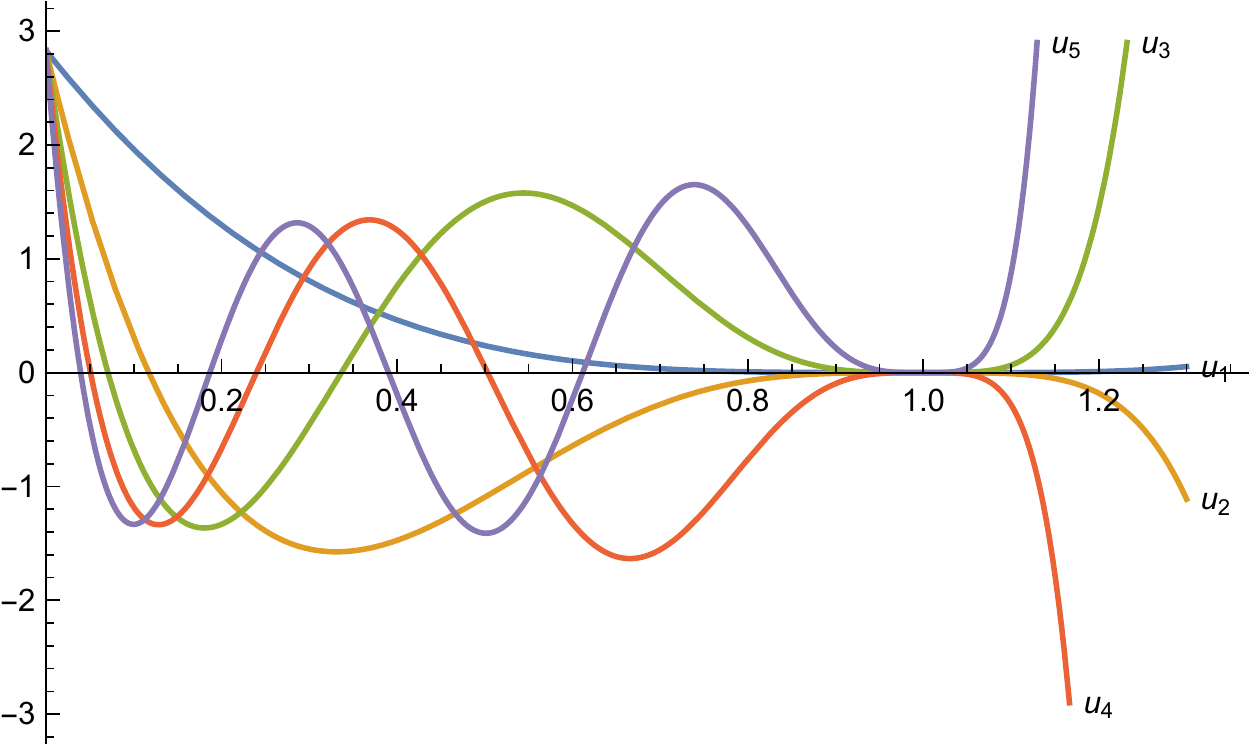}
  \end{center}
  \caption{Eigenfunctions $u_1,\dots,u_5$ for $n=4$}
  \label{fig:efs_4}
\end{figure}

\section{Distribution of the eigenvalues $\lambda_i$}

In Theorem~\ref{thm:detAn} we have stated that the determinant of $A_n(z)$ can
be written as a sum of expressions of the form $c\cdot\cosh(\alpha
z)\cdot\cos(\beta z)$ with $c\neq0$ and $\alpha,\beta\in\mathbb{R}$. The
location of the zeros of this determinant is mostly governed by the summand
whose $\cosh(\alpha z)$ term has the fastest asymptotic growth, i.e., whose
scaling factor~$\alpha$ is largest. Recall that the $\alpha$'s are obtained as
the real parts of sums (with $n$ summands) of different
$\omega_k=\exp\bigl(\frac{i\pi}{2n}(2k+[n]_2)\bigr)$. Figure~\ref{fig:rootsums}
shows all possible pairs $(\alpha,\beta)$ when $n=8$; note that there are
$\binom{16}{8}=12870$ sums, some of which add up to the same values.

If $n$ is even, then obviously $\{i,-i\}\subseteq\{\omega_k\mid
k=0,\dots,2n-1\}$.  Due to the ``correction term'' $[n]_2$, the same is
true for odd~$n$.  Hence, there are exactly $n-1$ numbers among the $\omega_k$
with strictly positive real part. By choosing them, plus one out of $\{i,-i\}$
(w.l.o.g. we choose~$i$), their sum will clearly exhibit the largest real part
among all $n$-sums of the~$\omega_k$.

We now study this sum in more detail. For even $n$ we obtain (by combining
pairs of complex conjugates):
\begin{align*}
  \sum_{k=-n/2+1}^{n/2} \!\!\!\! \omega_k
  &= \sum_{k=-n/2+1}^{n/2} \!\!\!\! \exp\bigl(i\tfrac{2\pi k}{2n}\bigr)
  = i + 1 + \! \sum_{k=1}^{n/2-1} 2\operatorname{Re}\bigl(\exp\bigl(i\tfrac{2\pi k}{2n}\bigr)\bigr) \\
  &= i + 1 + \! \sum_{k=1}^{n/2-1} 2\cos\bigl(\tfrac{2\pi k}{2n}\bigr)
  = i + \cot\bigl(\tfrac{\pi}{2n}\bigr).
\end{align*}
For odd $n$ we obtain
\begin{align*}
  \sum_{k=-(n-1)/2}^{(n-1)/2} \!\!\!\!\! \omega_k
  &= \sum_{k=-(n-1)/2}^{(n-1)/2} \!\!\!\!\! \exp\bigl(i\tfrac{2\pi k+\pi}{2n}\bigr)
  = i \, + \!\!\!\! \sum_{k=0}^{(n-1)/2-1} \!\!\!\! 2\operatorname{Re}\bigl(\exp\bigl(i\tfrac{2\pi k+\pi}{2n}\bigr)\bigr) \\
  &= i \, + \!\!\!\! \sum_{k=0}^{(n-1)/2-1} \!\!\!\! 2\cos\bigl(\tfrac{2\pi k+\pi}{2n}\bigr)
  = i + \cot\bigl(\tfrac{\pi}{2n}\bigr).
\end{align*}
Note that the result is the same for even and for odd~$n$. Hence the behavior
of $\det(A_n(z))$ is dominated by the term $c\cdot\cosh(\alpha
z)\cdot\cos(\beta z)$ with $\alpha=\cot\bigl(\tfrac{\pi}{2n}\bigr)$ and
$\beta=1$, as $z$ goes to infinity. This implies that for large~$z$ the zeros
of $\det(A_n(z))$ tend to the zeros of $\cos(z)$, which is exactly what was
observed in Section~\ref{se5} for the special case $n=2$, and in
Section~\ref{se6} for the cases $n=3$ and $n=4$.

\begin{figure}
  \begin{center}
    \includegraphics[width=0.9\textwidth]{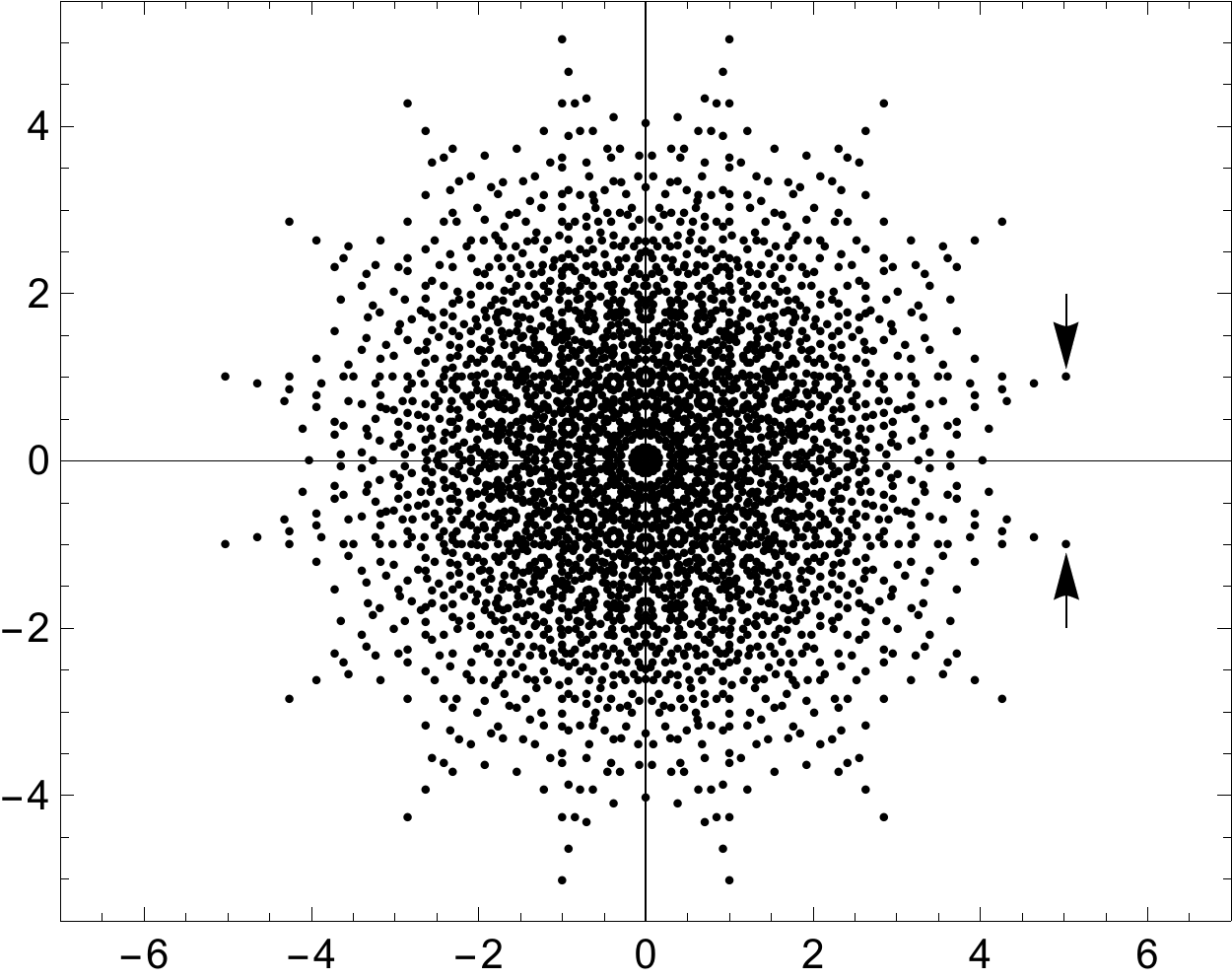}
  \end{center}
  \caption{The set \mbox{$\bigl\{\sum_{k\in I} \omega_k \mathrel{\big|}
    I\subseteq\{0,\dots,2n-1\},|I|=n\bigr\}$} plotted in the complex plane (for
    $n=8$); the two arrows point to those two numbers which dominate the
    asymptotics, i.e., with largest real part.}
  \label{fig:rootsums}
\end{figure}

\section{Conclusion}
In the paper, we have further investigated the Singular Value Decomposition of the
$n$-fold integration operators $J^n$. We have presented an algorithm that can be used to
compute the eigenvalues and eigenfunctions of $(J^n)^\ast J^n$. Out of the 5 steps of the algorithm, the first three can be done explicitely, but for the last two it seem impossible to solve them explicitely. The reason lies in the fact that the computation of the eigenvalues requires to solve a transcendental equation. However, for the cases $n=2,3,4$ we computed some of the eigenvalues as well as their eigenfunctions approximately. It needs to be mentioned that the numerical computation of the eigenvalues requires a computer program that allows computations in arbitrary precision - systems like MATLAB failed to deliver reasonable approximations to the eigenvalues.

\subsection*{Acknowledgments}
CK was supported by the Austrian Science Fund (FWF): P29467-N32 and F5011-N15.
RR was supported by the Austrian Science Fund (FWF):SFB F68-N36 and DK W1214.
BH was supported by German Research Foundation (DFG): HO 1454/12-1.

\end{document}